\documentclass{amsart}

\usepackage{amsmath,amssymb}
\usepackage{bm}
\usepackage{graphicx}
\usepackage{ascmac}
\usepackage{amsthm}
\usepackage{color} 
\usepackage{MnSymbol}

\theoremstyle{plain}
\newtheorem{thm}{Theorem}[section]

\newtheorem{prop}[thm]{Proposition}
\newtheorem{lmm}[thm]{Lemma}

\theoremstyle{definition}
\newtheorem{dfn}[thm]{Definition}
\newtheorem{set}[thm]{Setting and Notation}

\DeclareMathOperator{\an}{an}
\DeclareMathOperator{\sm}{sm}

\DeclareMathOperator{\NA}{NA}

\DeclareMathOperator{\Frac}{Frac}
\DeclareMathOperator{\Spec}{Spec}
\DeclareMathOperator{\Proj}{Proj}

\DeclareMathOperator{\ord}{ord}
\DeclareMathOperator{\PGL}{PGL}

\DeclareMathOperator{\red}{red}

\DeclareMathOperator{\T}{T}
\DeclareMathOperator{\Span}{Span}
\DeclareMathOperator{\Supp}{Supp}

\newcommand{\bA}{\mathbb{A}}
\newcommand{\bC}{\mathbb{C}}
\newcommand{\bD}{\mathbb{D}}

\newcommand{\bN}{\mathbb{N}}
\newcommand{\bP}{\mathbb{P}}
\newcommand{\bQ}{\mathbb{Q}}

\newcommand{\cJ}{\mathcal{J}}

\newcommand{\cO}{\mathcal{O}}

\usepackage[truedimen,top=25truemm,bottom=25truemm,left=25truemm,right=25truemm]{geometry}

\title{Meromorphic degeneration of rational functions over snc models of the projective line}
\author{Reimi Irokawa}
\address{NTT Institute for Fundamental Mathematics, NTT Communication Science Laboratories, Nippon Telegraph and Telephone Corporation, 3-9-11 Midori-cho, Musashino-shi, Tokyo 180-8686, Japan}
\email{reimi.irokawa@ntt.com}

\subjclass[2020]{Primary: 37F10, Secondary: 14G22, 37P30, 32A08}
\keywords{complex dynamics; non-archimedean dynamics; complex geometry}
\date{2025/1/31}

\begin{document}

\begin{abstract}
    For an analytic family $\{f_t\}_{t\in\bD^*}$ on the unit punctured disk that meromorphically degenerates at the origin, we show that its limiting measure on an snc model is given by the push forward of the canonical measure attached to the non-archimedean rational function naturally induced from the family, which is a generalization of the results by DeMarco-Faber and Okuyama.
\end{abstract}
\maketitle
\section{Introduction}

Let $\{f_t:\bP^1(\bC)\to\bP^1(\bC)\}_{t\in\bD^*}$ is an analytic family of rational functions of degree $d\geq2$, which degenerates at the origin, i.e., the limit function $f_0$ has a degree strictly less than $d$. In this paper, we study its degeneration feature.

We consider a family of probability measures $\{\mu_t\}$, where $\mu_t$ is the measure of maximal entropy, or the canonical measure, attached to $f_t$ for any $t\in\bD^*$. By Ma{\~n}e's Continuity theorem, the family of measures $\{\mu_t\}$ is weakly continuous as long as they do not degenerate. In the degenerate case, their limit is studied in several articles such as \cite{DeMarco:2005pd}, \cite{Kiwi_cubic}, \cite{Kiwi:2014nr}, \cite{DE-MARCO:2014tn}, \cite{Favre:2020rz}, \cite{article}, and \cite{okuyama2024intrinsicreductionsintrinsicdepths}. They use non-archimedean techniques to obtain the weak limit of the family $\{\mu_t\}$. That is, the family $\{f_t\}$ induces an analytic dynamics $f_{\bC((t))}^{\an}:\bP^{1,\an}_{\bC((t))}\to\bP^{1,\an}_{\bC((t))}$ by the Berkovich geometry \cite{berk1990}. Favre and Rivera-Letelier studied such dynamical systems \cite{Charles:2009be}, In particular, they constructed the invariant measure $\mu_f^{\NA}$, which has properties similar to the canonical measure in complex dynamics, and it plays a central role in the study of degeneration. Especially, DeMarco-Faber and Okuyama shows the following result:
\begin{prop}[\cite{DE-MARCO:2014tn}, Theorem B]\label{prop: intro - demarcofaber}
    Let $\{f_t\ |\ t\in\bD\}$ be a meromorphic family of rational functions of degree $d\geq2$ that degenerates at $t=0$. The family of measures $\{\mu_t\}_{t\in\bD^*}$ converges weakly on $\bP^1(\bC)$ to a countable sum of atomic measures $\mu_0$ as $t\to0$. The measure $\mu_0$ is given by the push-forward measure of $\mu_f^{\NA}$ by the reduction map $\red:\bP^{1,\an}_{\bC((t))}\setminus\{\zeta_{0,1}\}\to\bP^1(\bC)$.
\end{prop}
The definition of the reduction map will be presented in Section \ref{sec: sncmodels - background}. Our result is a generalization of this result; while they consider $\{\mu_t\}$ as a family of measures on $\bP^1(\bC)$, we may also regard $\mu_t$ as a measure on $\bP^1(\bC)\times \{t\}\subset \bP^1(\bC)\times \bD^*$. In this sense, their result describes the weak limit of the measures on $\bP^1(\bC)\times \bD$, the trivial model of $\bP^1(\bC)\times \bD^*$. We study the same weak limit, but over any snc model $X$ of $\bP^1(\bC)\times \bD^*$, which means $X|_{t\neq0}\simeq\bP^1(\bC)\times \bD^*$ and $X_0=X|_{t=0}$ has only simple normal crossing singularities. In this case, the reduction map $\red_{X}:\bP^{1,\an}_{\bC((t))}\setminus S\to X_0$ is defined as well as DeMarco-Faber and Okuyama's arguments, where $S$ is a finite set of points in $\bP^{1,\an}_{\bC((t))}$ determined from $X$ (for details; see Section \ref{sec: sncmodels - background}. The statement of our main result is as follows:
\begin{thm}
    Let $\{f_t\}_{t\in\bD^*}$ be an analytic family of degree $d\geq2$ which has genuenly bad reduction, and $X$ be an snc model of $\bP^1(\bC)\times \bD^*$. We regard the canonical measure $\mu_t$ attached to $f_t$ as the measure on $X_t$ for each $t$. Consider its induced non-archimedean dynamics $f_{\bC((t))}^{\an}$ and let $\mu_f^{\NA}$ be the canonical measure. Then, we have the following weak convergence of measures on $X$:
    \begin{align*}
        \lim_{t \to0}\mu_t= (\red_X)_*\mu_f^{\NA}.
    \end{align*}
\end{thm}

The notion of genuinely bad reduction is a slightly stronger condition than meromorphic degeneration at $t=0$, which will be presented in Definition \ref{definiton: background - good reduction}. The plan of this paper is as follows. We gather basic properties and facts for everything we need to show our main theorem in Section \ref{sec: sncmodels - background}, and show it in Section \ref{sec: proof of the main theorem}. In Section \ref{sec: example}, we see what happens to various snc models in the case of a family of quadratic polynomials.

\subsection*{Acknowledgment}
I would like to thank Prof. Yusuke Okuyama and Prof. Hiroyuki Inou for their useful discussions. I also appreciate Teppei Takamatsu's advice on algebraic geometric arguments.

\section{Berkovich geometry and reduction map} \label{sec: sncmodels - background}

In this section, we briefly recall the notions of snc models and their relation to Berkovich geometry in the case of the trivial family of the projective lines. General references for snc models are \cite{Shafarevich1966LecturesOM} and \cite{Artin:1971qr}, and those for Berkovich geometry are \cite{berk1990}.

\subsection{Berkovich analytification of the projective line}
Let $R=\cO_0$ be the set of one variable (complex) holomorphic functions defined around the origin of $\bC$, and $K=\Frac R$ be its fractional field, i.e., the set of one variable meromorphic functions defined around the origin. If we write a coordinate around the origin by $t$, taking the Laurent expansion around the origin enables us to embed $K$ to the formal Laurent series field $\bC((t))$ over $\bC$. We introduce the $t$-adic norm $|\cdot|$ on $\bC((t))$ as follows: 
\begin{align*}
    \left|\sum_{n=-m}^{\infty}a_nt^n\right|=\exp(-m),
\end{align*}
where $a_{-m}\neq0$ is the smallest degree with non-zero coefficient. $(\bC((t)),|\cdot|)$ is then a complete, non-trivial and non-archimedean valuation field. For algebraic varieties over $\bC((t))$, we may consider its {\it Berkovich analytification}, in this case, the Berkovich projective line $\bP_{\bC((t))}^{1,\an}$. As a topological space, the Berkovich projective line is given by the one-point compactification of the Berkovich affine line $\bA^{1,\an}_{\bC((t))}$. The Berkovich affine line $\bA^{1,\an}_{\bC((t))}$ is defined as the set of multiplicative seminorms on $\bC((t))[T]$ whose restriction to $\bC((t))$ equals the given $t$-adic norm $|\cdot|$. For a seminorm $x\in\bA^{1,\an}_{\bC((t))}$, we denote the evaluation at $f\in \bC((t))[T]$ by $|f(x)|$. This notation is compatible with the so-called {\it classical points}; any point $\alpha(t)\in\bC((t))$ defines a multiplicative seminorm $f\mapsto |f(\alpha(t))|$, and same type of norms can be defined for all the closed point of $\bP^1_{\bC((t))}$. Note that $\bP^{1,\an}_{\bC((t))}$ naturally contains the set of classical point; the point added in the compactification corresponds to $\infty\in\bP^1(\bC((t)))$.

A rational function $F\in \bC((t))(T)$ defines an (algebraic) endomorphism $F:\bP^1_{\bC((t))}\to \bP^1_{\bC((t))}$. By the analytification, $F$ also extends to the continuous map $F^{\an}:\bP^{1,\an}_{\bC((t))}\to\bP^{1,\an}_{\bC((t))}$ by $F(x)=x\circ F$ for $x\in\bP^{1,\an}_{\bC((t))}$. Especially, an analytic family $f_t(z):\bP^1(\bC)\times\bD^*\to \bP^1(\bC)\times\bD^*$ of rational functions on $\bP^1(\bC)$ of degree $d$ may be regarded as a rational function $f_K$ with coefficients in $K$, which is naturally embedded into $\bC((t))$. The analytification of rational function $f_K$ is denoted by $f_K^{\an}:\bP^{1,\an}_{\bC((t))}\to\bP^{1,\an}_{\bC((t))}$.

We remark that we can actually consider the analytic structure, the sheaf of analytic functions, on the Berkovich spaces, and we may regard $F^{\an}$ as an analytic morphism. We omit details of them here as we will not use them.

\subsection{Tree structure of Berkovich projective line}\label{subsec: sncmodels - background - treestructure}

One of the good properties of Berkovich curves is that they have a (limit of) graph structure. To see it, we first list up all the points belonging to $\bA^{1,\an}_{\bC((t))}$.

As written above, any closed points of $\bA^1_{\bC((t))}$ defines a point in $\bA^{1,\an}_{\bC((t))}$, which we call a type 1 point. For a closed point $x$ and $r>0$, set $\zeta_{x,r}\in\bA^{1,\an}_{\bC((t))}$ by
\begin{align*}
    |F(\zeta_{x,r})|:=\max_{|x-z|\leq r}|F(z)|,
\end{align*}
which is known to be a multiplicative norm on $K[T]$, i.e., a point corresponding to a closed disk centered at $x$ and radius $r$. We note that $z$ runs all the closed points. Since the $t$-adic norm extends uniquely to the algebraic closure of $\bC((t))$, the value $|z-x|$ is well-defined. We call $\zeta_{x,r}$ is of {\it type 2} if $r\in|\overline{\bC((t))}|=\exp(\bQ)$, where $|\overline{\bC((t))}|$ is the value group of the algebraic closure of $\bC((t))$, and of {\it type 3} otherwise. Also, for a sequence $\zeta_i=\zeta_{x_i,r_i}$ where $r_i>0$ is decreasing, its limit
\begin{align*}
    |F(\zeta)|:=\lim_{i\to\infty}|F(\zeta_i)|
\end{align*}
also defines a multiplicative seminorm. When $r_i\to0$, $\{x_i\}_i$ is a Cauchy sequence and there exists a limit $x$ on the completion of $\overline{\bC((t))}$, which we also call type 1. We note that the completion of the algebraic closure of $\bC((t))$ is the minimal algebraically closed field containing it, whose norm is uniquely determined from that of $\bC((t))$. When a decreasing sequence of closed disks has an empty intersection, it defines a point neither of type 1,2 nor 3. Such a limit point $\zeta$ is called type 4. It is known that there are no other points on $\bA^{1,\an}_{\bC((t))}$.

By the ultrametric inequality, for any pair of two closed disks, exactly one of the following holds: one includes another, or they have no intersection. This makes $\bP^{1,\an}_{\bC((t))}$ into a tree structure; in particular, the space $\bP^{1,\an}_{\bC((t))}$ is uniquely path connected. For any $x$ of type 2, we set $\T_x(\bP^{1,\an}_{\bC((t))})$ be the set of connected components of $\bP^{1,\an}_{\bC((t))}\setminus \{x\}$, called the {\it tangent space} at $x$. Especially, the tangent space at $\zeta_{0,1}$ has a natural bijection to $\bP^1(\bC)$, since every component of $\bA^{1,\an}_{\bC((t))}\setminus \{\zeta_{0,1}\}$ must contain at most one constant by construction, and the unique component which does not contain any constant corresponds to $\infty\in\bP^1(\bC)$. 

\subsection{Non-archimedean dynamics}
We briefly recall the dynamics over the non-archimedean projective line. Let $f_t(z)\in \bC((t))[z]$ be a rational function of degree $d\geq2$. The analytification $f^{\an}$ of $f$ acts on $\bP^{1,\an}_{\bC((t))}$, which defines a dynamical system. For $f$, we introduce the notion of good reduction:
\begin{dfn}\label{definiton: background - good reduction}
    We say $f_t$ has {\it good reduction} if $\deg f_0=d$, and {\it potentially good reduction} if there exists a finite field extension $L$ of $\bC((t))$ and $\tau\in\PGL(2,K)$ such that $f_{t,\tau}:=\tau\circ f\circ \tau^{-1}\in L(z)$ has good reduction. We say $f_t$ has a {\it genuinely bad reduction} if $f_t$ does not have a potentially good reduction. 
\end{dfn}
We note that when $f_t$ has good reduction, $f_t$ extends to $f_t:\bP^1_{\bC[[t]]}\to\bP^1_{\bC[[t]]}$ while the extension has indeterminacy otherwise. Whether or not the map $f_t$ has a good reduction, the fundamental properties of its dynamical system are studied by Rivera-Letelier and Favre \cite{Charles:2009be}. Especially the canonical measure is defined in the same way as the complex case, and it detects maps with potentially good reduction:
\begin{prop}\label{prop: background - basic facts from non-archimedean dynamics}
    For any $f_t(z)\in\bC((t))(z)$ of degree $d\geq2$, there exists a unique probability measure $\mu_f^{\NA}$ such that 
    \begin{itemize}
        \item $\mu_f^{\NA}$ has no mass on the exceptional set of $f$;
        \item $(f^{\an})_*\mu_f^{\NA}=\mu_f^{\NA}$;
        \item $(f^{\an})^*\mu_f^{\NA}=d\cdot\mu_f^{\NA}$.
    \end{itemize}
    The Julia set of $f$ coincides with the support of $\mu_f^{\NA}$. Also, when $f_t(z)$ has good reduction, $\mu_f^{\NA}=\delta_{\zeta_{0,1}}$. Moreover, $\mu_f^{\NA}$ has a mass on type 2 points if and only if $f$ has a potentially good reduction. In this case, i.e., there exists a type 2 point $\zeta$ such that  $\mu_f^{\NA}=\delta_{\zeta}$, the conjugacy of $f$ by $\tau\in \PGL(2,\overline{\bC((t))})$ has good reduction as long as $\tau$ maps $\zeta$ to $\zeta_{0,1}$.
\end{prop}

\subsection{Snc models and reduction map}
\begin{dfn}
    An {\it snc model} $X$ of $\bP^1(\bC)\times \bD^*$ is a complex manifold having a proper submersion $X\to\bD$ such that $X|_{t\neq0}\simeq \bP^1(\bC)\times\bD^*$ and $X|_{t=0}$ defines a simple normal crossing divisor.
\end{dfn}
By abuse of notation, we sometimes regard $X$ as a regular scheme over $\Spec R$ and we identify $\bP^1_K$ and $\bP^1_R$ with $\bP^1(\bC)\times\bD^*$ and $\bP^1(\bC)\times\bD$ respectively. For an snc model $X$ of $\bP^1(\bC)\times\bD^*$, we write
\begin{align*}
    X_0=\sum_{i=1}^{k}n_i E_i.
\end{align*}
We may attach a so-called {\it divisorial point} $\eta_i\in\bP^{1,\an}_{\bC((t))}$ to each $E_i$, defined as 
\begin{align*}
    |f(\eta_i)|=\exp(-\ord_{E_i}(f)/n_i),
\end{align*}
where $\ord_E(f)$ is the order of the vanishing of $f$ at the generic point of $E$. 
 
For an snc model $X$, the {\it reduction map} is defined as follows: for any closed point $x\in X_{\bC((t))}\simeq\bP^1_{\bC((t))}$, its topological closure $\overline{\{x\}}\subset X_{\bC[[t]]}$ with respect to Zariski topology contains unique closed point on the special fiber $X_0$ by the valuative criterion of properness. We denote it by $\red_X(x)$. By 1.3 of \cite{baker-payne-rabinoff}, one can extend this map to the analytification:
\begin{prop}
    The reduction map $\red_X$ extends to an anti-continuous map $\red_X:X_{\bC((t))}^{\an}\to X_{0}$, which maps $\eta_i$ to the generic point of $E_i$ and the other points to the closed points of $X_0$. Moreover, it is constant on each connected component of $X_{\bC((t))}^{\an}\setminus\{\eta_i\}_{i=1}^{k}$.
\end{prop}

In the later discussion, we use the symbol $\red_X$ for the reduction map restrected to $\bP^{1,\an}_{\bC((t))}\setminus\{\eta_i\}_{i=1}^{k}$, i.e., $\red_X:\bP^{1,\an}_{\bC((t))}\setminus\{\eta_i\}_{i=1}^{k}\to X_0(\bC)$. 

We remark that for a trivial model $X=\bP^1(\bC)\times\bD$, the corresponding reduction map is written in terms of the tangent space. By definition, the divisorial point corresponding to the unique divisor on the special fiber $E=X_0$ is nothing but $\zeta_{0,1}$. Then, the natural bijection $\bA^{1,an}_{\bC((t))}\setminus\{\zeta_{0,1}\}\to\bP^1(\bC)$ in Section \ref{subsec: sncmodels - background - treestructure} naturally extends to a bijection $\bP^{1,an}_{\bC((t))}\setminus\{\zeta_{0,1}\}\to\bP^1(\bC)$ and this is the reduction map $\red_X$ in this case.

We use this reduction map to define a map from the space of probability measures on $\bP^{1,\an}_{\bC((t))}$ without any mass on $\{\eta_i\}_{i=1}^{k}$ to that of probability measures that is written as a countable sum of atomic measures; if $\rho$ is a probability measure on $\bP^{1,\an}_{\bC((t))}$ with no mass on $\{\eta_i\}_{i=1}^{k}$, i.e. $\rho(\bP^{1,\an}_{\bC((t))}\setminus \{\eta_1,\ldots,\eta_k\})=1$, 
\begin{align*}
    (\red_X)_*\rho=\sum_{x\in X_0(\bC)}\rho(U_x)\delta_x,
\end{align*}
where $U_x=\red_X^{-1}(\{x\})$, is a probability measure on $X_0(\bC)\simeq X|_{t=0}$ which is a countable sum of atomic measures.

\subsection{contraction of curves and relation between minimal models}
An snc model $X$ of $\bP^1(\bC)\times\bD^*$ is said to be {\it minimal} if $X|_{t=0}\simeq \bP^1(\bC)$. 
There is a classical result of Castelnuovo claiming that any (-1)-curve $C\subset X_0$ is contracted by blow-down; that is, there exists another snc model $X'$ and a birational morphism $\pi: X\to X'$ such that $\pi$ is isomorphic outside of $C$ and the image of $C$ is contracted to one point. 

As a consequence of intersection theory and the above Castelnuovo's theorem, we have the following fact:
\begin{prop}
    For any snc model $X$ of $\bP^1(\bC)\times\bD^*$ and any irreducible component $E\subset X|_{t=0}$, we have an isomorphism $E\simeq \bP^1(\bC)$. There exists a composite of contraction of curves $\pi:X\to X_E$ such that $X_E$ is another snc model of $\bP^1(\bC)\times\bD^*$ with $X_E|_{t=0}=E$.
\end{prop}

There is a natural minimal model $X=\bP^1(\bC)\times\bD$, the trivial model of $\bP^1(\bC)\times\bD^*$. Actually, the model is unique in the following sense:

\begin{prop}\label{prop: extension of PGL(2,K)}
    For any minimal model $X$, there exists $\tau\in \PGL(2,K)$ such that $\tau$ extends to an isomorphism $\tau:X\to\bP^1(\bC)\times\bD$.
\end{prop}

\begin{proof}
    We fix a homogeneous coordinate $([x:y],t)\in \bP^1\times\bD$ so that $\bP^1_R=\Proj R[x,y]$. It is enough to find $\tau\in\PGL(2,K)$ such that $\tau:\bP^1_R\simeq X$. We put $\pi: X\to\Spec R$ to be the structural morphism. Take any rational point $x\in X(K)$, which always exists as $X_K\simeq\bP^1_K$, and put $\overline{x}$ as its closure in $X$. We first see that $M:=\Gamma(\Spec R, (\pi)_*\cO_X(\overline{x}))$ is a rank $2$ free $R$-module. Since $X_K\simeq\bP^1_K$ and $X_{\bC}\simeq \bP^1_{\bC}$, $\Gamma(\Spec K, (\pi)_*\cO_X(\overline{x}))$ (resp. $\Gamma(\Spec \bC, (\pi)_*\cO_X(\overline{x}))$) is a rank 2 free $K$-module (resp. $\bC$-module). By flat base change, they coincide $M_K$ and $M_{\bC}$, which shows the claim.

    Take any basis $z,w\in M$ over $R$. Since they generate $M$, there exists an immersion $\tau: X\to \bP^1_R$ such that $z$ maps to $x$ and $w$ maps to $y$. Since it is an isomorphism on the generic fiber, $\tau\in\PGL(2.K)$. Since it is also an isomorphism on the special fiber, $\tau$ is an isomorphism on whole $X$.
\end{proof}

For a blow-up (or blow-down) morphism of snc models, we have the following relation in their reduction maps:

\begin{lmm}\label{lmm: background - contraction - compatibility}
    Let $X$ be an snc model dominating another snc model $Y$ with contraction map $\pi: X\to Y$. Since $\pi$ is an isomorphism on their generic fibers, it induced $\pi^{\an}:X^{\an}_{\bC((t))}\simeq Y^{\an}_{\bC((t))}$ Then, we have
    \begin{align*}
        \pi\circ \red_X=\red_Y\circ \pi^{\an} : \bP^{1,\an}_K\setminus \{\eta_1,\ldots,\eta_k\}\to Y_0,
    \end{align*}
    where $X_0=\bigcup_{1\leq i\leq k} E_i$ and $\eta_i$ is a divisorial point associated to $E_i$.
\end{lmm}
\begin{proof}
    Since every component of $X_{\bC((t))}^{\an}\setminus\{\eta_1,\ldots,\eta_k\}$ contains a closed point and the reduction map is constant on the components, it is enough to show the equality only on the set of closed points. Then, this result follows from the uniqueness of the reduction map; that is, for any closed point $x$ of $X_{\bC((t))}$, its reduction under $\red_X$ is a unique point on the special fiber of $\overline{\{x\}}\subset X_{\bC[[t]]}$. Therefore, $\pi\circ \red_X$ is a unique point on the special fiber of $\pi(\overline{\{x\}})\subset Y_{\bC[[t]]}$. On the other hand, $\red_Y(\pi(x))$ is a unique point on the special fiber of $\overline{\pi(x)}\subset Y_{\bC[[t]]}$. Since $\overline{\pi(x)}\subset\pi(\overline{\{x\}})$ and they both contain exactly one point on their special fibers, they must coincide.
\end{proof}

\section{Proof of the main theorem}\label{sec: proof of the main theorem}
We now specify the situation we work on based on what we saw in Section \ref{sec: sncmodels - background}.

\begin{set}\label{setting: merom degeneration}
    We fix an inhomogeneous coordinate $z\in\bP^1(\bC)$ and $t\in\bD$. Let $f_t(z):\bP^1(\bC)\times\bD^*\to\bP^1(\bC)\times\bD^*$ be an analytic family of rational functions of degree $d$ meromorphically degenerating at the origin, i.e.,
    \begin{align*}
        f_t(z)=\frac{a_0(t)z^d+\cdots+a_d(t)}{b_0(t)z^d+\cdots+b_d(t)},
    \end{align*}
    where $a_i,b_j:\bD\to\bC$ are holomorphic functions so that $\deg f_t=d$ for any $t\in\bD^*$. We assume that $f$ has genuinely bad reduction defined in Definition \ref{definiton: background - good reduction}.

    Let $X$ be any snc model of $\bP^1(\bC)\times\bD^*$ such that
    \begin{align*}
        X|_{t=0}=\bigcup_{i=1}^k E_i,
    \end{align*}
    where each $E_i$ is irreducible. By taking the base change $K\to \bC((t))$ and the analytification in the sense of Berkovich, we consider an analytic variety $\bP^{1,\an}_{\bC((t))}$ and its analytic endomorphism $f^{\an}_{\bC((t))}:\bP^{1,\an}_{\bC((t))}\to \bP^{1,\an}_{\bC((t))}$. For each $i=1,\ldots,k$, the divisorial point corresponding to $E_i$ is denoted by $\eta_i\in\bP^{1,\an}_{\bC((t))}$. The reduction map is denoted by $\red_X:X^{\an}_{\bC((t))}\setminus\{\eta_1,\ldots,\eta_k\}\to X|_{t=0}$. When $X$ is a trivial model $\bP^1(\bC)\times\bD$, the reduction map $\red_X$ is simply denoted by $\red$.

    For each $t\in\bD^*$, let $\mu_t$ be the canonical measure attached to $f_t$, which we regard as a probability measure on $\bP^1(\bC)\times\{t\}$. Let $\mu_f^{\NA}$ be the canonical measure attached to $f^{\an}_{\bC((t))}$, which is a probability measure on $\bP^{1,\an}_{\bC((t))}$. 
\end{set}
The main theorem is, with the above notation, as follows:

\begin{thm} \label{thm: proof of the main thm - main thm}
    In Setting and Notation \ref{setting: merom degeneration}, any weak limit of $\{\mu_t\}$ is a countable sum of atomic measures. Moreover, the following weak convergence of measures holds on $X$:
    \begin{align*}
        \lim_{t\to0}\mu_t=(\red_X)_*\mu_f^{\NA}.
    \end{align*}
\end{thm}

For $i=1,\ldots,k$, let $\pi_i: X\to X^{(i)}$ be the contraction of curves on $X|_{t=0}$ so that $X^{(i)}$ is a minimal model of $\bP(\bC)\times \bD^*$ such that $X^{(i)}|_{t=0}=E_i$. We put $\red_i:=\red_{X_i}:\bP^{1,\an}_{\bC((t))}\setminus\{\eta_i\}\to\bP^1(\bC)$, which coincide with $\pi_i\circ \red_X$ on $\bP^{1,\an}_{\bC((t))}\setminus\{\eta_1,\ldots,\eta_k\}$ by Lemma \ref{lmm: background - contraction - compatibility}. 

When $X$ is minimal, Theorem \ref{thm: proof of the main thm - main thm} is a slight transformation of the result by DeMarco-Faber and Okuyama \cite{DE-MARCO:2014tn}, \cite{article}, and \cite{okuyama2024intrinsicreductionsintrinsicdepths}:
\begin{lmm}\label{lemma: main thm - main thm for minimal models}
    If $X$ is minimal, Theorem \ref{thm: proof of the main thm - main thm} is true, i.e., 
    \begin{align*}
        \lim_{t\to0}\mu_t=(\red_X)_*\mu_f^{\NA}.
    \end{align*}
\end{lmm}
\begin{proof}
    By Proposition \ref{prop: extension of PGL(2,K)}, there exists $\tau\in\PGL(2,K)$ such that it extends $\tau:\bP^1(\bC)\times\bD\to X$ as an isomorphism. Set $f_{\tau^{-1}}:=\tau^{-1}\circ f\circ\tau:\bP^1(\bC)\times\bD\dashedrightarrow \bP^1(\bC)\times\bD$. Note that the canonical measure attached to $f_{\tau^{-1}}$ is written as $(\tau^{-1})_*\mu_t$, and that attached to $(f_{\tau^{-1}})_{\bC((t))}^{\an}$ is written as $(\tau^{-1,\an}_{\bC((t))})_*\mu_f^{\NA}$. By Proposition \ref{prop: intro - demarcofaber}, we obtain the following weak convergence on $\bP^1(\bC)\times\bD$:
    \begin{align*}
        \lim_{t\to0}(\tau^{-1})_*\mu_t=(\red)_*(\tau^{-1,\an}_{\bC((t))})_*\mu_f^{\NA}.
    \end{align*}
    Since $\tau$ gives an isomorphism between $\bP^1(\bC)\times\bD$ and $X$ and the divisorial point corresponding to the trivial model is the Gauss point $\zeta_{0,1}$, we obtain $\tau_{\bC((t))}^{-1,\an}(\zeta_{0,1})=\eta_i$, which implies $\tau\circ\red\circ\tau_{\bC((t))}^{-1,\an}=\red_X$ on $X$. Since $\tau$ is an isomorphism, we can push forward everything to obtain
    \begin{align*}
        \lim_{t\to0}\mu_t &=\lim_{t\to0}(\tau)_*(\tau^{-1})_*\mu_t\\
        &=(\tau)_*(\red)_*(\tau^{-1,\an}_{\bC((t))})_*\mu_f^{\NA} \\
        &=(\tau\circ\red\circ\tau_{\bC((t))}^{-1,\an})_*\mu_f^{\NA} \\
        &=(\red_X)_*\mu_f^{\NA},
    \end{align*}
    which completes the proof.
\end{proof}

Now we complete the proof of Theorem \ref{thm: proof of the main thm - main thm}.
\begin{proof}[Proof of Theorem \ref{thm: proof of the main thm - main thm}]
    By Lemma \ref{lemma: main thm - main thm for minimal models}, the statement is true for each $X^{(i)}$. We put $\mu_{E_i}:=(\red_i)_*\mu_f^{\NA}$ be the limit of $\{\mu_t\}$ on $X^{(i)}$. That is,
    \begin{align}\label{eq: main thm - main thm - limit for push forward to minimal models}
        \lim_{t\to0}(\pi_i)_*\mu_t=\mu_{E_i}.
    \end{align}
    We first show that any limit $\rho$ of $\{\mu_t\}_{t\in\bD^*}$ on $X$ is a countable sum of atoms. Suppose the contrary and we assume $\rho|_{E_i}$ is not a countable sum of atoms, nor is $\rho|_{E^{\sm}_i}$ since the singular locus of $E_i$ is a finite set, which is contradiction since $(\pi_1)_*\mu_{E_i}|_{E_i^{\sm}}=\mu_{E_i}|_{E_i^{\sm}}$. 

    Next, we show the weak convergence. Put $\rho$ to be any weak limit of $\{\mu_t\}$ on $X$. Since the relations $(\pi_i)_*\rho=\mu_{E_i}$ uniquely determines $\rho$, it is enough to check $(\red_X)_*\mu_f^{\NA}$ satisfies them. Indeed, by Lemma \ref{lmm: background - contraction - compatibility}, $\pi_i\circ\red_X=\red_i\circ \pi_{i,\bC((t))}^{\an}$. The genuenly bad reduction implies $\mu_f^{\NA}$ has no mass on $\{\eta_1,\ldots,\eta_k\}$, i.e.,  $(\pi_i)_*(\red_X)\mu_f^{\NA}=\mu_{E_i}$. 
\end{proof}
The measure $(\red_X)_*\mu_f^{\NA}$ is written explitly in terms of $\{\mu_{E_i}\}$:
\begin{align*}
    \rho(\{z\})=\begin{cases}
            &\mu_{E_i}(\{z\})\text{ if }z\in E_i^{\sm}; \\
            &\mu_{E_i}(\{z\})+\mu_{E_j}(\{z\})-1\text{ if }z\in E_i\cap E_j. 
    \end{cases}
\end{align*}

\section{Example}\label{sec: example}
We consider a family of quadratic polynomials
\begin{align*}
    f_t(z)=\frac{1}{t}(z^2+1),
\end{align*}
which degenerates to $f_t(z)=\infty$ as $t\to0$. Note that this polynomial is given by conjugating the polynomial $z^2+1/t^2$ so that the two fixed points on $\bC$ stay bounded for any $t$; in particular, $\mu_t$ is the Cantor measure for any $t\in\bD^*$. Let $x_+$ and $x_-$ be the fixed points of $f_t$, i.e.,
\begin{align*}
    x_{\pm}(t)=\frac{t\pm\sqrt{t^2-4}}{2}.
\end{align*}
To see the degenerating limit of various snc models, we first recall the non-archimedean dynamics of $f$.
\subsection{Non-archimedean dynamics of quadratic polynomials}
For simplicity, we take the completion of the algebraic closure $L$ of $\bC((t))$, and we work on dynamics over $L$ instead of $\bC((t))$. The dynamics of quadratic polynomials over $L$ are studied in \cite{Benedetto2007}. Note that they work on a slightly different parametrization, so we present their result in our parametrization. The points $x_{\pm}$ may be regarded as elements in $\bA^{1,\an}_L$, which are fixed by $f_L^{\an}$ and $\red(x_{\pm})=x_{\pm}(0)=\pm\sqrt{-1}$.

Especially in this case, the Julia set of $f_L^{\an}$ is a 2-Cantor set, on which the action is conjugate to the shift map, just as in the complex case. $x_{\pm}$ define points on $\bA^{1,\an}_L$, which are fixed by $f_L^{\an}$. We set
\begin{align*}
    &B:=\left\{z\in \bA^{1,\an}_L\middle| |z|\leq1 \right\}; \\
    &U_{\pm}:=\left\{z\in \bA^{1,\an}_L\middle| |z-x_{\pm}|<1 \right\}.
\end{align*}
Note that in this case $U_{\pm}\subset B$ and $U_{\pm}=\red^{-1}\{x_{\pm}\}=\red^{-1}\{\pm \sqrt{-1}\}$. 
Both of $B_{\pm}$ split into two strictly smaller balls if we take the inverse image of $f_L^{\an}$, and we obtain
\begin{align*}
    \cJ(f_L^{\an})=\bigcap_{n=0}^{\infty}(f_L^{\an})^{-n}B,
\end{align*}
where $\cJ(f_L^{\an})$ is the Julia set of $f_L^{\an}$. We put
\begin{align*}
    B_{\pm}:=\left\{z\in\bA^{1,\an}_L\middle| |z-x_{\pm}|\leq|t|\right\}.
\end{align*}
Then $f_L^{\an}$ maps $B_{\pm}$ onto $B$. By definition, $B_{\pm}\subset U_{\pm}$. Therefore, the above equation of the Julia set is re-written as
\begin{align*}
    \cJ(f_L^{\an})=\left(\bigcap_{n=0}^{\infty}(f_L^{\an})^{-n}U_{+}\right)\cup\left(\bigcap_{n=0}^{\infty}(f_L^{\an})^{-n}U_{-}\right).
\end{align*}
The conjugate is given as follows: for any $z\in \cJ(f_L^{\an})$, $h(z)=(\alpha_n)_{n,\geq0}$, where
\begin{align*}
    \alpha_n=\begin{cases}
        &+\quad\text{if}\quad f_L^{n,\an}(z)\in U_+;\\
        &-\quad\text{if}\quad f_L^{n,\an}(z)\in U_-.
    \end{cases}
\end{align*}
By the above arguments, $(f^{\an}_L)^{-n}B$ contains exactly $2^n$ balls and each has the same mass with respect to $\mu_f^{\NA}$.

\subsection{Weak limits of the canonical measures}
If we take the trivial model $\bP^1(\bC)\times\bD$, the weak limit of $\mu_t$ is a sum of atoms $(\delta_{\sqrt{-1}}+\delta_{-\sqrt{-1}})/2$. Actually, this is the only minimal model such that the weak limit of $\mu_t$ does not contract to an atom; indeed, for any $\tau\in\PGL(2, K)\setminus \PGL(2, R)$, $f_{\tau}$ has fixed point $\tau(x_{\pm})$ with $|\tau(x_{\pm})|>1$ or $|\tau(x_{+})-\tau(x_{-})|<1$. In the former case the Julia set of $f_L^{\an}$ is inside of $U_{\infty}=\{|z|>1\}$, which implies the limiting measure is $\delta_{\infty}$. In the latter case it is inside of $U_{x_{+}(0)}:=\{|z-x_+|<1\}$, which coincides with $U_{x_-(0)}$, and the limiting measure is $\delta_{x_{+}(0)}$.

In general, we have the following claim for the limiting measure to be contracted to an atom over minimal models:
\begin{prop}\label{prop: condition on non-triviality of limiting measures}
    In the Notation same as \ref{setting: merom degeneration}, For any $z\in\bP^{1,\an}_L$, we put $X^{(z)}$ be the minimal model of $\bP^1(\bC)\times\bD^*$ such that the corresponding divisorial point is $z$ and $\rho_z=(\red_{X^{(z)}})_*\mu_f^{\NA}$ be the weak limit of $\{\mu_t\}_{t\in\bD^*}$. We define the set $S$ as follows:
    \begin{align*}
        S:=\left\{z\in\bP^{1,\an}_L\ \middle|\ z\text{ is of type }2 \text{ and }\Supp\rho_z\text{ has at least two points}\right\}.
    \end{align*}
    Then, we obtain the following relation:
    \begin{align*}
        \overline{S}=\overline{\Span \cJ(f_L^{\an})}.
    \end{align*}
\end{prop}
Here, for a subset $E\subset \bP^{1,\an}_L$, we denote by $\Span E$ the smallest connected set containing $E$ and by $\overline{E}$ the topological closure of $E$. 
\begin{proof}
    Note that $E$ cannot be singleton on $\eta$ by Proposition \ref{prop: background - basic facts from non-archimedean dynamics} as we assume $f_t$ has a genuinely bad reduction. 

    We first show $\overline{S}\subset \overline{\Span\cJ(f^{\an}_L)}$. Take any $z\in S$. By conjugating $f_L$ by some $\tau\in\PGL(2,L)$ such that $f(z)=\zeta_{0,1}$, we may assume $z=\zeta_{0,1}$. We take two distinct complex numbers $c_1, c_2\in\bC$ such that $\rho_z(\{c_1\}),\rho_z(\{c_2\})\neq0$. Then, by denoting $U_{c_i}:=\red^{-1}\{c_i\}$, we see that $\Supp\mu_f^{\NA}\cap U_{c_i}\neq\emptyset$. Then, since $\bP^{1,\an}_L$ is uniquely path connected, any connected set which contains subsets of $U_{c_1}$ and $U_{c_2}$ must path through $z$, i.e., $z\in\Span(\cJ(f^{\an}_L)$. Then, we obtain $S\subset \Span\cJ(f^{\an}_L)$, and therefore $\overline{S}\subset\overline{\Span\cJ(f^{\an}_L)}$.

    It remains to show that $\Span\cJ(f^{\an}_L)\subset \overline{S}$. Take any point $z\in \Span\cJ(f^{\an}_L)$. Since we assume $f$ has a genuinely bad reduction, there exists a point $z'\in \cJ(f^{\an}_L)$ such that $z\neq z'$ and $[z,z']\subset \Span\cJ(f^{\an}_L)$, where $[z,z']$ denotes the unique path connecting $z$ and $z'$. We take a sequence $z_n\to z$ inside $[z,z']$ such that $z_n$ is of type 2 for any $n$. Then, for each $n$, there exists two distincs points $c_1,c_2\in X^{(z_n)}_0$ such that $\red_{X^{(z_n)}}^{-1}(c_i)\cap \Span\cJ(f^{\an}_L)\neq\emptyset$, i.e., $z_n\in S$. Since $z_n$ converges to $z$, we obtain $z\in \overline{S}$.
\end{proof}

Now, we go back to the example $f_t(z)=(z^2+1)/t$ and consider limiting measures over various non-minimal snc models. By the above Proposition, non-trivial examples should occur inside $\overline{\Span\cJ(f^{\an}_L)}$. For any $\alpha=(\alpha_k)_{k=1}^n\in\{\pm\}^n$, we put $\eta_\alpha$ to be the type 2 point corresponding to closed disk $B_{\alpha}$ such that $h^k(B_\alpha)\in B_{\alpha_1}$. Then, $\overline{\Span\cJ(f^{\an_L})}=\Span\cJ(f^{\an_L})$ in this case is a binary tree of infinite length with the branch points $\{\eta_{\alpha}\ |\ \alpha=(\alpha_k)_{k=1}^n\in\{\pm\}^n, n\in\bN\}$ and the endpoints $\cJ(f^{\an}_L)$. 

We already show that the limiting measure is $\rho_0:=(\delta_{\sqrt{-1}}+\delta_{-\sqrt{-1}})/2$. We set $\{X^{(n)}\}$ to be a sequence of snc models so that 
\begin{align*}
    X^{(n)}_0=E+\sum_{k=1}^n\sum_{\alpha\in\{\pm\}^k}E_{\alpha},
\end{align*}
where $E$ is the special fiber of the trivial model and $E_{\alpha}$ is an irreducible component with corresponding divisorial point $\eta_{\alpha}$. Such models are given by compositions of blow-ups starting from the trivial model, and we obtain a birational map $\pi_n: X^{(n+1)}\to X^{(n)}$ which contracts $E_{\alpha}$ with $\alpha\in\{\pm\}^{n+1}$ for each $n$. Denote $\rho_n$ by the limiting measure of $\{\mu_t\}$ on $X^{(n)}$. Then, $\rho_n$ is supported on $\bigcup_{\alpha\in\{\pm\}^n} E_{\alpha}$ and a sum of $2^n$ distinct atoms of the same mass $1/2^n$ on the set of points on the image by $\pi_n$ of its strict transform.

\bibliographystyle{plain}
\bibliography{merom-degeneration-of-p1}
\end{document}